\documentclass[12pt]{amsart}
\usepackage{epsf}
\usepackage{amsbsy,amsmath}
\usepackage{mathtools}
\usepackage{mathrsfs}
\usepackage{amsfonts}
\usepackage{amssymb}
\usepackage{enumitem}
\usepackage{eucal}
\usepackage{enumitem}
\usepackage{hyperref}
\usepackage{graphics,mathrsfs}
\usepackage{amsthm}
\usepackage{xcolor}
\newtheorem{theorem}{Theorem}[section]
\newtheorem{lemma}[theorem]{Lemma}
\newtheorem{proposition}[theorem]{Proposition}

\theoremstyle{definition}
\newtheorem{definition}[theorem]{Definition}

\newtheorem{conjecture}[theorem]{Conjecture}

\theoremstyle{remark}
\theoremstyle{claim}
\newtheorem{remark}[theorem]{Remark}
\numberwithin{equation}{section}

\numberwithin{equation}{section}
\newcommand{\subscript}[2]{$#1 _ #2$}
\newsavebox{\savepar}

\begin{document}
	%	\begin{frontmatter}		
		\title[Neumann problem with a discontinuous nonlinearity]{Neumann problem with a discontinuous nonlinearity} 		
		\author[D. Choudhuri]{Debajyoti Choudhuri}
		\address[D. Choudhuri]{School of Basic Sciences, Indian Institute of Technology Bhubaneswar, Khordha, 752050, Odisha, India}
		\email{dchoudhuri@iitbbs.ac.in}		
		\author[D.D. Repov\v{s}]{Du\v{s}an D. Repov\v{s}}
		\thanks{{\it Corresponding author}: Du\v{s}an D. Repov\v{s}: dusan.repovs@guest.arnes.si}
		%$^{\footnote{xxx}$
		\address[D.D. Repov\v{s}]{Faculty of Education and Faculty of Mathematics and Physics, University of Ljubljana \&
		Institute of Mathematics, Physics and Mechanics,
		 1000 Ljubljana, Slovenia}
		\email{dusan.repovs@guest.arnes.si}		
		\author[K. Saoudi]{Kamel Saoudi}
		\address[K. Saoudi]{Basic and Applied Scientific Research Center, Imam Abdulrahman Bin Faisal University, Saudi Arabia}
		\email{kmsaoudi@iau.edu.sa}		
	%\today
		\maketitle	

	\begin{abstract}
		\noindent This study is devoted to proving the existence of weak solutions for a nonlinear elliptic problem with Neumann-type boundary data. The problem is driven by a discontinuous power nonlinearity and a nonsmooth prescribed data. Additionally, we aim to derive an estimate that proves the well-posedness of the problem. This estimate serves as an evidence for the uniqueness of the existing solution when the boundary term is ``smooth".\\
		\begin{flushleft}
			{\sc Math. Subj. Classif. (2020)}:~35J35, 35J60, 35R11, 35J75, 46E35.
		\end{flushleft}
	\end{abstract}
	
	% \tableofcontents		
	
	\section{Introduction}\label{introduction}
	In this work, we investigate the existence of solutions to an elliptic equation under Neumann boundary conditions in Euclidean space. The problem under study is formulated as follows:
	\begin{equation}\label{main}
		\left\{ 
		\begin{array}{ll}
		     -\Delta_pu=H(u-b)|u|^{q-2}u, \qquad \text{a.e. in}~\Omega,\\
		     \qquad u>0,~\qquad  \qquad  \qquad \qquad \qquad~~\text{a.e. in}~\Omega,\\
		     -|\nabla u(x)|^{p-2}\frac{\partial}{\partial n}u(x)\in\partial f(u(x)),~\   \text{a.e. on}~\partial\Omega.
		\end{array} 
		\right.
	\end{equation}
    where $\Omega$ is a bounded domain in $\mathbb{R}^N$ with $N \geq 2$. The exponents $p$ and $q$ are chosen such that $2 \leq p < q < p^*$, where $p^*$ denotes the critical Sobolev exponent, i.e.
    $$p^*=\frac{Np}{N-p}
    \hbox{ if }
    p<N,
    \hbox{ and }
    p^*=+\infty
    \hbox{ if }
    p\ge N.$$
     The function $H$ represents the Heaviside function, which depends on a positive parameter $b$. On the boundary $\partial\Omega$, the operator $\frac{\partial}{\partial n}$ denotes the outward normal derivative, and the locally Lipschitz function $f$ specifies the given boundary data.	
	
	It is worth noting that a similar problem can be found in Cen et al. \cite{cen1} and in Papalini's work \cite{papalini1}, although the nonlinearity differs in these cases.
	Compared with \cite{cen1}, we allow a locally Lipschitz boundary potential $f$, leading to a genuinely nonsmooth (possibly multivalued) Neumann condition expressed via the Clarke subdifferential; in this setting we also obtain a well-posedness estimate, and uniqueness when $\partial f$ is single-valued. In addition, Theorem~\ref{existence} treats a singular term that, to the best of our knowledge, is not covered in \cite{cen1,papalini1}.
	
Problem \eqref{main} represents a partial differential equation that describes the behavior of a function $u$ within the domain $\Omega$. The presence of the Neumann boundary condition $$-|\nabla u|^{p-2}\frac{\partial}{\partial n}u(x)\in\partial f(u(x)),$$ which holds almost everywhere on the boundary $\partial\Omega$, introduces additional complexity to the problem.	  
 Problem \eqref{main}, 
 with $$-|\nabla u(x)|^{p-2}\frac{\partial}{\partial n}u(x)$$ being a single value,
  has been extensively investigated
   and significant results have been obtained in previous studies. 

Problems of the type \eqref{main} have been studied through variational techniques such as the Mountain Pass Lemma and the Nehari manifold method. Classical results, including those by Ambrosetti and Badiale \cite{AmBadiale}, Drabek-Poho\v{z}aev \cite{DrPo}, and Perera-Pinchover \cite{Le}, provide existence and multiplicity theorems for nonlinear $p$-Laplacian equations.
   
   For more information on these findings and related subjects, we encourage readers to explore the research by Ambrosetti et al. \cite{AmBrCe} and Diening et al. \cite{DiHaHaRu}. Their work offers significant insights into the theory of elliptic equations influenced by nonlinearities and
    under the
    Neumann boundary conditions. Furthermore, the studies by Saoudi et al. \cite{saoudi1} and Santos et al. \cite{santos1}, along with their references, provide additional context on contemporary trends in elliptic problems featuring discontinuous nonlinearities.

 This article establishes the existence of infinitely many solutions to the non-autonomous elliptic problem \eqref{main} featuring a discontinuous nonlinearity of Heaviside type. The absence of differentiability prevents the application of standard variational methods, and the
 absence 
 of controlling parameters introduces 
 additional
  challenges. To overcome these obstacles, we develop a framework based on nonsmooth critical point theory, demonstrating that the problem exhibits a
   diverse
   solution structure despite its irregular nature. 
   
   Our findings provide new insights into the behavior of solutions $ u $ within the domain $ \Omega$, particularly in relation to the boundary data $ f $, and extend known results for problems with smooth nonlinearities.    
 These approaches, well-suited for handling problems with nonsmooth or discontinuous nonlinearities, enable us to effectively address the challenges posed by both the equation and 
    the Neumann boundary conditions. 
 We demonstrate the utility of these results by studying a singular nonlinear elliptic free boundary value problem
 with a prescribed nonsmooth Neumann data on the boundary $\partial\Omega$.
 
 The main conditions
  \ref{C1}-\ref{C3}
 (resp.
  \ref{H1} and \ref{H2}), 
  under which we shall study the problem, 
    will be defined in Section \ref{aux_res}
    (resp. Section \ref{prelim}). We now state our main results.    

  \begin{theorem}\label{multiplicity}
Assume that conditions  \ref{C1}-\ref{C3} are satisfied. Then the boundary value problem
\begin{equation}\label{mod_main}
		\left\{ \begin{array}{ll}-\Delta_pu=H(|u|-b)|u|^{q-2}u,~\qquad  ~~\text{a.e. in}~\Omega,\\			
			-|\nabla u(x)|^{p-2}\frac{\partial}{\partial n}u(x)\in\partial f(u(x)),~~~\text{a.e. on}~\partial\Omega.
		\end{array} 
		\right.
	\end{equation}
 possesses infinitely many distinct solutions.
\end{theorem} 
 
 \begin{theorem}\label{unique}
    If $ \partial f(u) = \{v\} $, then  problem \eqref{main} is well-defined, and the solution $ u $ is unique.
\end{theorem}
 		
	\begin{remark}
The term `well-defined' used in Theorem \ref{unique} implies that the solution $u$ depends continuously on the boundary data $f$. This property, known as continuous dependence, is a crucial component of well-posedness.
\end{remark}
	
	\begin{remark}\label{CJ}
	The set of critical points of any real valued $C^1$-functional $\Phi$ defined over a Banach space will be denoted by $\mathcal{C}_{\Phi}$.
	\end{remark}
    
To illustrate the utility of Theorems \ref{multiplicity} and \ref{unique}, we shall also establish the following result.	
\begin{theorem}\label{existence}
    Assume that conditions \ref{H1} and \ref{H2} are satisfied. Then there exists a solution of the following problem
    \begin{equation}\label{main_singular}
        \left\{ 
        \begin{array}{ll}
        -\Delta_p u = H(u - b)|u|^{q-2}u + \lambda u^{-\gamma}, & \text{a.e. in } \Omega, \\
        u > 0, & \text{a.e. in } \Omega, \\
                -|\nabla u(x)|^{p-2}\frac{\partial}{\partial n}u(x) \in \partial f(u(x)), & \text{a.e. on}~ \partial \Omega. 
        \end{array} 
        \right.
    \end{equation}
    This solution exists in the function space $ W^{1,p}(\Omega), $ for every $2 \leq p < q < p^*.$
\end{theorem}
					 
	Theorem \ref{existence} is the prime novelty of this work, since we  treat a singular nonlinearity which to the best of our knowledge, has not been considered in the nonsmooth setup of problem \eqref{main}.
	
	We shall conclude the introduction by stating two interesting related conjectures (for a comprehensive source on variable Sobolev spaces, the readers may refer to Papageorgiou et al.  \cite{PRR}).
	
		\begin{conjecture}
			 Theorems \ref{multiplicity} and \ref{existence} hold for problems \eqref{main} and \eqref{main_singular} with a variable exponent $q$.
			\end{conjecture}			
			\begin{conjecture}
			   Problem \eqref{main} yields the same result when considered over a metric measure space or a Heisenberg group. 
		\end{conjecture}

	\section{Preliminaries}\label{prelim}
  In this section, we shall present a concise overview of the definitions and fundamental properties of nonsmooth theory. These concepts will be vital for the discussion of problem \eqref{main}. We begin by introducing the nonsmooth Palais-Smale condition, which is essential for the analysis of critical points.
\begin{definition}(\cite[Definition $2$]{radulescu1})[Nonsmooth Palais-Smale condition]\label{def10}
Let $J: X \to \mathbb{R}$ be a locally Lipschitz functional. We say that $J$ satisfies the $(\text{PS})_c$ condition if any sequence $(u_n)$ in $X$ with the following properties
\begin{itemize}
\item[$(\mathrm{i})$] $J(u_n) \to c;$ 
\item[$(\mathrm{ii})$] $D_J(u_n) \to 0;$
\end{itemize}
possesses a subsequence which
 converges strongly in $X$.
\end{definition}
The object $D_{J}(u)$ is the nonsmooth representative of the derivative of a nondifferentiable functional, say $J$, 
and it is defined as follows
 $$D_{J}(u):=\min\{\|v\|_{X^*}:v\in\partial J(u)\}.$$
Here, $\partial J(u)$ denotes the generalized gradient (subdifferential) of $J$ in the sense of Clarke \cite{saoudi1}. 

A fundamental tool for locating critical points of such nondifferentiable functionals is the following extension of the Mountain Pass Theorem.
\begin{theorem}(R\u{a}dulescu \cite[Corollary $1$]{radulescu1})\label{MP_them}
 Let $X$ be a Banach space and $J: X \to \mathbb{R}$ a locally Lipschitz functional such that $J(0) = 0$. Assume there exist constants $\rho_1, \rho_2 > 0$ and an element $\sigma \in X$ satisfying
 \begin{enumerate}
     \item $J(u) \geq \rho_2$ for all $u \in X$ with $\|u\|_X = \rho_1;$ 
     \item $\|\sigma\|_X > \rho_1$ and $J(\sigma) < 0.$
 \end{enumerate}
 Define the Mountain Pass value
 $$
 c = \inf_{\zeta \in \Gamma} \max_{t \in [0,1]} J(\zeta(t)), \quad \text{where} \quad \Gamma = \big\{ \zeta \in C([0,1]; X) : \zeta(0) = 0,\, \zeta(1) = \sigma \big\}.
$$
Then there exists a sequence $(u_n) \subset X$ such that:
 $$
 J(u_n) \to c \quad \text{and} \quad D_J(u_n) \to 0, \quad \text{as } n \to \infty,
$$ 
 and $c \geq \rho_2$. 
 
 \noindent
 Furthermore, if $J$ satisfies the nonsmooth Palais–Smale condition at level $c$, then $c$ is a critical value of $J$; that is, there exists $u \in X$ such that $J(u) = c$ and $0 \in \partial J(u)$.
 \end{theorem}
 
More on nonsmooth analysis can be found in Grossinho-Tersian \cite{grossinho1}.
Next, we present an important result which relates the convergence of a sequence in the function space $X$ to the convergence of its corresponding subgradients in the dual space $X^*$.

\begin{proposition}(Santos-Tavares \cite[Proposition $1$]{santos1})\label{res1}
Let $(u_n)\subset X$ and $(\eta_n)\subset X^*$ with $\eta_n\in\partial J(u_n)$. If $u_n\to u$ in $X$ and $\eta_n\overset{\ast}{\rightharpoonup}\eta$ in $X^*$ (w.r.t. weak*-convergence), then $\eta\in\partial J(u)$.
\end{proposition}

To investigate the solutions of the Neumann boundary value problem \eqref{main}, we define an energy functional that helps us find the critical points related to these solutions.

\begin{definition}\label{def11}
A function $u \in W^{1,p}(\Omega)$ is said to be a solution of the Neumann problem \eqref{main} if it is a critical point of the energy functional $I: W^{1,p}(\Omega) \to \mathbb{R}$ defined by
\begin{align}\label{eq1}
I(u):=\frac{1}{p}\int_{\Omega}|\nabla u|^p\,dx-\int_{\Omega}G(u)\,dx+\int_{\partial\Omega}f(Tr(u))\,dS, 
\end{align}
where $\mathrm{Tr}$ denotes the trace operator, $dS$ being the surface element.
Here, $G:\mathbb{R}\to\mathbb{R}$ is the (even) potential associated with the discontinuous nonlinearity, defined by
\[
G(t):=\frac{1}{q}H(|t|-b)\big(|t|^q-b^q\big).
\]
It is locally Lipschitz on $\mathbb{R}$, $G\in C^1(\mathbb{R}\setminus\{\pm b\})$, and
$$
\hbox{ for every } 
|t|\neq b,
\hbox{ we have }
G'(t)=H(|t|-b)|t|^{q-2}t.$$

Since $q<p^*$, the Sobolev embedding $W^{1,p}(\Omega)\hookrightarrow L^q(\Omega)$ guarantees that the bulk term $\int_\Omega G(u)\,dx$ is finite for every $u\in W^{1,p}(\Omega).$
In particular, $I$ is well defined and locally Lipschitz on $W^{1,p}(\Omega)$.
\end{definition}
We also observe that \begin{align}\label{obs1}\underset{u\in W^{1,p}(\Omega)}\inf I(u)\leq 0=I(0).\end{align}

We now state some important properties of the functional 
$$ Z:L^q(\Omega)\to\mathbb{R}, \ \ Z(u):=\int_{\Omega}G(u)\,dx.$$
Set $$h(t):=H(|t|-b)|t|^{q-2}t
\hbox{ for }
t\in\mathbb{R}.$$
 The following lemma describes the Clarke subdifferential of $Z$ (cf. Santos--Figueiredo \cite[Lemma 3.1]{santos2}).

\begin{lemma}\label{subderivative}
The functional $Z: L^q(\Omega) \to \mathbb{R}$, defined by $$Z(u):=\int_{\Omega}G(u)\,dx,$$
 is locally Lipschitz. Moreover, for every $u\in L^q(\Omega)$ and every $\eta\in\partial Z(u),$ we have $\eta\in L^{\frac{q}{q-1}}(\Omega)$ and
\[
\eta(x)\in [\underline{h}(u(x)),\overline{h}(u(x))],\quad \text{for a.e. }x\in\Omega,
\]
where
\[
[\underline{h}(t),\overline{h}(t)]= \begin{cases}
\{0\}, & |t|<b,\\
\big[\min\{0,|t|^{q-2}t\},\max\{0,|t|^{q-2}t\}\big], & |t|=b,\\
\{|t|^{q-2}t\}, & |t|>b.
\end{cases}
\]
\end{lemma}

\begin{proof}
Fix $R>0$ and let $u,v\in L^q(\Omega)$ satisfy $\|u\|_{L^q},\|v\|_{L^q}\le R$. Since
\[
G(t)=\frac{1}{q}\big(|t|^q-b^q\big)_+\quad\text{for every }t\in\mathbb{R},
\]
where $s\mapsto (s)_+$ is $1$-Lipschitz, we have 
\[
|G(s)-G(t)|\le \frac{1}{q}\,\big||s|^q-|t|^q\big|\le \big(|s|^{q-1}+|t|^{q-1}\big)|s-t|,
\hbox{ for every }
s,t\in\mathbb{R}.
\]
Therefore,
\[
|Z(u)-Z(v)|\le \int_{\Omega}\big(|u|^{q-1}+|v|^{q-1}\big)|u-v|\,dx
\le \big(\|u\|_{L^q}^{q-1}+\|v\|_{L^q}^{q-1}\big)\,\|u-v\|_{L^q},
\]
which shows that $Z$ is locally Lipschitz on $L^q(\Omega)$.

Next, $G$ is locally Lipschitz on $\mathbb{R}$ and $\partial G(t)=[\underline{h}(t),\overline{h}(t)]$ (in the sense of Clarke) with $[\underline{h}(t),\overline{h}(t)]$ as stated in the lemma. By Clarke's theorem on generalized gradients of integral functionals (see, e.g., Clarke \cite[Theorem 2.7.5]{clark1}), for every $u\in L^q(\Omega)$ and every $\eta\in\partial Z(u),$ we have $\eta\in L^{\frac{q}{q-1}}(\Omega)$ and
\[
\eta(x)\in \partial G\big(u(x)\big)=[\underline{h}(u(x)),\overline{h}(u(x))],\quad\text{for a.e. }x\in\Omega.
\]
This proves the claim.
\end{proof}

In this paper, we assume that the function $f: \mathbb{R} \to \mathbb{R}$ fulfills the following conditions
\begin{enumerate}[label=(\subscript{H}{{\arabic*}})]
    \item \label{H1} $f \in L^1(\mathbb{R})$;
    \item \label{H2} for every $L > 0$, there exists $P_L \in L_+^{p'}(\partial\Omega)$ such that for each $w(u) \in \partial f(u(x))$, it follows that $|w(u(x))| \leq P_L(x)$ almost everywhere on $\partial\Omega$ for $|u(x)| < L$.
\end{enumerate}

Additionally, according to Lebourg's theorem (see Clark \cite[Theorem 2.3.7, p. 41]{clark1}), and since the locally Lipschitz condition in $\mathbb{R}$ is equivalent to the Lipschitz condition on bounded sets, there exists $\gamma \in L^{p'}_+(\partial\Omega)$ such that 
\begin{align}\label{H3}
    |f(t)| \leq \gamma(x) |t|, \quad \text{a.e. on } \partial\Omega \text{ for every } t \in \mathbb{R}. 
\end{align}

	We observe that  conditions \ref{H1} and \ref{H2} are also sufficient to tackle problems driven by a discontinuous nonlinearity in addition to an irregular Neumann data.

\begin{remark}\label{rem3}
A direct consequence of Lemma \ref{subderivative} is that every
 element $\eta$ of the subdifferential $\partial Z(u)$ satisfies $\eta \in L^{\frac{q}{q-1}}(\Omega)$ and, for almost every $x \in \Omega$, the inclusion $\eta(x) \in [\underline{h}(u(x)), \overline{h}(u(x))]$ holds.
\end{remark} 
 
\section{Auxiliary  Results}\label{aux_res}

In  the next section, we shall establish our main results through nonsmooth theory. Accordingly, we shall outline and prove some auxiliary results in this section.

\begin{proposition}\label{PS_Cond}
Under hypothesis \ref{H2}, the energy functional $I: W^{1,p}(\Omega) \to \mathbb{R}$ defined in \eqref{eq1} satisfies the Palais–Smale condition.
\end{proposition}
\begin{proof}
Assume that $(u_n) \subset W^{1,p}(\Omega)$ is a sequence such that 
$$I(u_n) \to c \quad \text{and} \quad \partial I(u_n) \to 0 \quad \text{in } W^{-1,p'}(\Omega).$$ 
We shall first show that $I$ is coercive. 

To this end, define $$J: W^{1,p}(\Omega) \to \mathbb{R}$$ by 
$$J(u) := \int_{\partial\Omega} f(Tr(v)) dS,$$ 
where 
$$Tr: W^{1,p}(\Omega) \to W^{1-\frac{1}{p},p}(\partial\Omega)$$ 
is the trace operator. 

Let $$K(u) := p^{-1} \int |\nabla u|^p dx$$ and define $T: L^p(\partial\Omega) \to \mathbb{R}$ by 
$$T(v) := \int_{\partial\Omega} f(v) dS, \quad \text{for every } v \in L^p(\partial\Omega).$$ 

Notably, since $T$ is locally Lipschitz, we can apply Clarke's theorem \cite[Theorem 2.7.5]{clark1} to obtain 
$$\partial T(u) = \{v \in L^{p'}(\partial\Omega): v \in \partial f(u(x)) \text{ a.e. on } \partial\Omega\}.$$ 
Since
 $J = T \circ Tr$, the Chain Rule leads to 
$$\partial J(v) = \partial T(Tr(v)) DTr(v).$$
Now, let 
$$Tr^*: \left(W^{1-\frac{1}{p},p}(\partial\Omega)\right)^* \to \left(W^{1,p}(\Omega)\right)^*$$ 
be the adjoint operator corresponding to $Tr$. According to Clarke \cite[Remark 2.3.11]{clark1}, we have 
$$\partial J(v) \subset Tr^*(\partial T(Tr(v))), \quad \text{for every } v \in W^{1,p}(\Omega).$$

We shall use 
$$\|u\| := \left[\left(\int_{\Omega} |\nabla u|^p dx\right)^p + \left(\int_{\Omega} |u|^p dx\right)^p\right]^{\frac{1}{p}}$$ 
as the norm on elements of $W^{1,p}(\Omega)$.

We are now ready to define the weak formulation of problem \eqref{main} as follows:
$$I(u) := K(u) - A(u) + F(u),$$
where 
$$A(u) := \frac{1}{q} \int_{\Omega} H(|u| - b)(b^q - |u|^q) dx, \quad F(u) = \int_{\partial\Omega} f(Tr(u)) dS.$$
Thus, we have 
$$\partial I(u) \subset \{K'(u)\} - \partial A(u) + \partial F(u), \quad \text{for every } u \in W^{1,p}(\Omega).$$

Suppose that the sequence $(I(u_n))$ is bounded by $\beta > 0$. Then we have
\begin{align}\label{bddness1}
\begin{split}
I(u_n) = & K(u_n) - A(u_n) + F(u_n) \leq \beta.
\end{split}
\end{align}
Consequently, it follows that
\begin{align}\label{bddness2}
\begin{split}
\frac{1}{p} \|u_n\|^p \leq & \beta + \frac{1}{q} \int_{\Omega \cap \{|u_n| > \alpha\}} (b^q - |u_n|^q) dx + \int_{\partial\Omega} f(Tr(u(x))) dS.
\end{split}
\end{align}
Therefore, we obtain
\begin{align}\label{bddness3}
\begin{split}
\frac{1}{p} \leq & \frac{\beta}{\|u_n\|^p} + \frac{1}{q\|u_n\|^p} \int_{\Omega \cap \{|u_n| > \alpha\}} (b^q - |u_n|^q) dx + \frac{1}{\|u_n\|^p} \int_{\partial\Omega} f(Tr(u(x))) dS\\
\leq & \left(\frac{\beta}{\|u_n\|^p} + \frac{1}{\|u_n\|^p} \int_{\partial\Omega} f(Tr(u(x))) dS\right) \to 0, \quad \text{as } n \to \infty,
\end{split}
\end{align}
since by \eqref{H3} there exists $\gamma \in L^{p'}(\partial\Omega)$ such that 
$$|f(t)| \leq \gamma(x) |t| \quad \text{a.e. on } \partial\Omega \text{ and for every } t \in \mathbb{R}.$$
Hence, we have
\begin{align}\label{bddness4}
\begin{split}
\frac{1}{\|u_n\|^p}\left|\int_{\partial\Omega} f(Tr(u(x))) dS\right| \leq & \frac{1}{\|u_n\|^p} \int_{\partial\Omega} \gamma |Tr(u)| dS \leq \frac{C \|\gamma\|_{p'}}{\|u_n\|^{p-1}} \to 0, \quad \text{as } n \to \infty.
\end{split}
\end{align}
This leads to a contradiction and hence $I$ is indeed coercive.

 The coercivity of $I$ ensures the boundedness of the sequence $(u_n)$. This allows us to extract a subsequence, still denoted as $(u_n)$, which converges weakly to some $u \in W^{1,p}(\Omega)$.

We first observe that 
$$\|u\|^p + o(1) = D_{I(u_n)}(u),$$
and thus there exist sequences $(a_n)$ and $(\theta_n)$ such that 
$$a_n \rightharpoonup a \quad \text{in } L^{q'}(\Omega), \quad \text{and} \quad \theta_n \rightharpoonup \theta \quad \text{in } L^{p'}(\partial\Omega).$$
Consequently, we have 
$$\|u\|^p = \int_{\Omega} a u \, dx + \int_{\partial\Omega} Tr^*(\theta(u)) \, dx.$$

Note that   
\begin{equation}\label{convergence1}
\|u\|^p = \int_{\Omega} a u \, dx + \int_{\partial\Omega} Tr^*(\theta(u)) \, dx + o(1) = D_{I(u_n)}(u_n) = \underset{n\to\infty}{\lim} \|u_n\|^p,
\end{equation}  
which implies 
$$u_n \to u, \text{ as } n \to \infty  \text{ in } W^{1,p}(\Omega).$$ 
Therefore, the (PS) condition indeed holds.
 \end{proof}
 
 \begin{proposition}\label{critical_points}
 If  conditions \ref{H1} and \ref{H2} hold, then the functional $I,$ defined in \eqref{eq1}, has a nontrivial critical point in $X$.
 \end{proposition}
 
 \begin{proof}
 We invoke a fundamental result from nonsmooth critical point theory due to Chang \cite[Theorem 3.5]{chang1}, which states that every locally Lipschitz functional on a reflexive Banach space that is bounded below attains its infimum as a critical value. The functional $I$ is locally Lipschitz and, by its coercivity, is bounded from below. Since $W^{1,p}(\Omega)$ is a reflexive Banach space, Chang's theorem guarantees that $I$ possesses a critical point at the level $\inf_{X} I$. This concludes the proof of Proposition \ref{critical_points}.
 \end{proof} 

Next, we state the following result.
\begin{theorem}\cite[Theorem $3.1$]{dai1}\label{fab1}
 Let $ X $ be a reflexive and separable Banach space. Then there exists a biorthogonal system $\{(e_n, e_n^*)\}_{n=1}^{\infty}$ such that the sequence $\{e_n\}$ is complete in $X$, the sequence $\{e_n^*\}$ is complete in the dual space $X^*$, and they satisfy the duality relation $\langle e_n^*, e_m \rangle = \begin{cases} 
 	0, & \text{if}~ n\neq m \\
 	1, & \text{if}~ n=m,
 \end{cases}$ 
 for all $n, m \in \mathbb{N}.$
 \end{theorem}

Finally, we recall  the Fountain Theorem for nonsmooth functionals. 
Let 
$$X_k=\text{span}\{e_k\}, \ Y_k=\overset{k}{\underset{j=1}\bigoplus}X_j,  \ Z_k=\overset{\infty}{\underset{j=k}\bigoplus}X_j.$$
\begin{theorem}(Alves et al. \cite[Theorem $3.8$]{alves1})\label{FT}
    Let $X$ be a Banach space and consider the function $\psi : X \to \mathbb{R}$, which is even and locally Lipschitz. Suppose that the sets $X_k, Y_k, Z_k$ are defined as in the previous sections. Then, there exist constants $R_k > r_k > 0$ such that the following conditions are satisfied
    \begin{enumerate}[label=(\subscript{C}{{\arabic*}})]
        \item\label{C1} as \( k \to \infty \), we have 
        $$\underset{\underset{u \in Z_k}{\|u\| = r_k}}{\inf} \psi(u) \to \infty;$$
        \item\label{C2} for large \( k \), it holds that 
        $$\underset{\underset{u \in Y_k}{\|u\| = R_k}}{\inf} \psi(u) \leq 0;$$
        \item\label{C3} the functional $\psi$ satisfies the $(PS)_c$ condition for all \( c > 0 \).
    \end{enumerate}
    Under these conditions, the function $\psi$ possesses an unbounded sequence of critical values.
\end{theorem}

\section{Proofs of Main Results}\label{proof}

\begin{proof}[Proof of Theorem \ref{multiplicity}]
    We begin by noting that condition \ref{C3} is satisfied according to Proposition \ref{PS_Cond}. From \eqref{H3}, we can derive the following inequality:
    $$I(u) \leq \frac{1}{2} \|u\|^2 - \frac{1}{q} \|u\|_q^q + \|\gamma\|_2 \|Tr(u)\|_2 \leq \frac{1}{p} \|u\|^p - \frac{1}{q} \|u\|_q^q + c \|\gamma\|_2 \|u\|.$$

    In any finite-dimensional space, all norms are equivalent. By selecting $\|u\| = \rho_k$ to be sufficiently large, we obtain $I(u) \leq 0$, thus confirming that condition \ref{C2} holds.

    Let us define 
    $$a_k = \underset{\underset{\|u\| = 1}{u \in Z_k}}{\sup} \|u\|_2.$$ 
    It is evident that $a_k \to 0,$ as \( k \to \infty \).

    Moreover, we have
    $$I(u) \geq \frac{1}{2} \|u\|^2 - \frac{1}{q} \|u\|_q^q - c \|\gamma\|_2 \|u\| \geq \frac{1}{2} \|u\|^2 - \frac{1}{q} c' a_k \|u\|^q - c \|\gamma\|_2 \|u\|.$$

    By choosing 
    $$r_k = (c' a_k)^{\frac{1}{2-q}}, \quad \text{for each } u \in Z_k, \ \|u\| = r_k,$$
    we arrive at:
    $$I(u) \geq \left(\frac{1}{2} - \frac{1}{q}\right)(c' a_k)^{\frac{2}{2-q}} - c \|\gamma\|_2 \|u\| \to \infty.$$
    
    This confirms that condition \ref{C1} is satisfied as well. Therefore, we conclude that there exists an unbounded sequence of critical values, implying that the problem 
    \eqref{main} indeed has infinitely many solutions.
\end{proof}

		\begin{proof}[Proof of Theorem \ref{unique}]
			Let $u_1, u_2$ be any two solutions
			of problem \eqref{main},
			 i.e. let $0\in\partial I(u_1), \partial I(u_2)$. Then we have 
			\begin{align}\label{unique1}
				\begin{split}
					0=& \int_{\Omega}|\nabla u_1|^{p-2}\nabla u_1\cdot\nabla v dx-\int_{[u_1\neq\alpha]}\left(-H(u_1-b)u_1^{q-1}\right)v dx-\int_{[u_1=b]}\theta_1 v dx+\int_{\partial\Omega}\eta_1 v dS\\
					0=&  \int_{\Omega}|\nabla u_2|^{p-2}\nabla u_2\cdot\nabla v dx-\int_{[u_2\neq b]}\left(-H(u_2-b)u_2^{q-1}\right) v dx-\int_{[u_2=b]}\theta_2 v dx+\int_{\partial\Omega}\eta_2 v dS.
				\end{split}
			\end{align}
			Here, $\theta_1, \theta_2\in [\underline{g}(u),\bar{g}(u)],$ as given in Lemma \ref{subderivative}.
			On subtracting the first equation from the second one in \eqref{unique1} and taking $\phi:=u_1-u_2,$ we obtain
			\begin{align}\label{unique2}
				\begin{split}
					0=& \int_{\Omega}(|\nabla u_1|^{p-2}\nabla u_1-|\nabla u_2|^{p-2}\nabla u_2)\cdot\nabla(u_1-u_2) dx\\
					&+\int_{{[u_1\neq b, u_2\neq b]}}(H(u_1-b)u_1^{q-1}-H(u_2-b)u_2^{q-1})(u_1-u_2) dx+\int_{\partial\Omega}(\eta_1-\eta_2)(u_1-u_2) dS\\
					\geq & C\int_{\Omega}|\nabla(u_1-u_2)|^{p} dx+\int_{[u_1\neq\alpha, u_2\neq b]}\left(H(u_1-b)u_1^{q-1}-H(u_2-b)u_2^{q-1}\right)(u_1-u_2) dx\\
					&+\int_{\partial\Omega}(\eta_1-\eta_2)(u_1-u_2)dS.
				\end{split}
			\end{align}							
			We now have  to consider the following four cases:
%			\begin{enumerate}\label{cases_to_discuss}
				(a)~$u_1>b$, $u_2>b$, (b)~$u_1>b$, $u_2<b$, (c)~$u_1<b$, (d)~$u_2<b$, $u_1<b$, $u_2>b$.
%			\end{enumerate}			
			Therefore we have 
			\[ \int_{\Omega}(H(u_1-b)u_1^{q-1}-H(u_2-b)u_2^{q-1})(u_1-u_2) dx\begin{cases} 
				\geq C_q\|u_1-u_2\|_{L^q(\Omega)}^q, & \text{if}~u_1>b, u_2>b \\
				=0, & \text{if}~ u_1<b, u_2<b\\
				=-\int_{\Omega}u_2^{q-1}(u_1-u_2)dx, & \text{if}~u_1<b, u_2>\alpha\\
				=\int_{\Omega}u_1^{q-1}(u_1-u_2)dx,	& \text{if}~u_1>b, u_2<b
			\end{cases}
			\]
			which is a positive term, so by invoking \eqref{unique2}, we obtain
			\begin{align}\label{estq}
				\begin{split}
					C\|\eta_1-\eta_2\|_{p'}^{1/p}\geq &\|u_1-u_2\|.
				\end{split}
			\end{align}
			This gives an estimate which proves that the problem is indeed well-defined.
			 Furthermore, if $\partial f(u)=\{v\}$, then the solution is unique.
		\end{proof}

\section{An Example}\label{example}

 We begin by considering the  functional energy associated with the problem \eqref{main_singular}:
 \begin{align}\label{funct_sing}
 I_S(u) = \frac{1}{p} \int_{\Omega} |\nabla u|^p  dx - \int_{\Omega} G(u)  dx + \int_{\partial\Omega} f(\mathrm{Tr}(u))  dS - \frac{\lambda}{1-\gamma} \int_{\Omega} |u|^{1-d}  dx.
 \end{align}
 
 The first difficulty is that the singular term $|u|^{1-d}$ prevents $I_S$ from being Fréchet differentiable, even when the other nonlinearities are smooth. To overcome this, we introduce a truncation. Let $\underline{u}$ be a subsolution of the problem \eqref{main_singular}. We define the truncated function
$$
 \tilde h(x, u(x)) = 
 \begin{cases} 
     |u(x)|^{-d-1}u(x), & \text{if } |u(x)| > \underline{u}(x), \\
     \underline{u}(x)^{-d}, & \text{if } |u(x)| \leq \underline{u}(x).
 \end{cases}
$$
 This modification yields a functional with improved regularity, allowing us to apply nonsmooth variational methods.
 
 We now establish a key property: any positive solution $u$ of \eqref{main_singular} satisfies $u \geq \underline{u}$ almost everywhere in $\Omega$. This implies that the critical set of the singular functional, $\mathcal{C}_{I_S}$, is contained within the critical set of the regularized functional, $\mathcal{C}_I$. We follow the arguments in Ghosh et al. \cite[Lemma $3$]{ghosh1}.
 
 To prove this, let $u > 0$ be a solution of \eqref{main_singular} and define the set $A = \{x \in \Omega : u(x) \leq \underline{u}(x)\}$. Consider the difference between the weak formulations of \eqref{main_singular} and the equation for the subsolution $\underline{u}$. Testing this difference with $\phi = (u - \underline{u})_-$ (the negative part) yields:
 \begin{equation}\label{dj1}
 \begin{split}
 0 &\geq \left\langle -\Delta_p u + \Delta_p \underline{u}, (u - \underline{u})_- \right\rangle \\
   &= \lambda \int_{\Omega} \left( u^{-d} - \underline{u}^{-d} \right) (u - \underline{u})_-  dx + \int_{\Omega} u^{q-1} (u - \underline{u})_-  dx \geq 0.
 \end{split}
 \end{equation}
 The first inequality follows from the monotonicity of the $p$-Laplacian operator. The final integrand is non-positive on $A$ because $u \leq \underline{u}$ implies $u^{-d} \geq \underline{u}^{-d}$ and $(u - \underline{u})_- \leq 0$. The second integral is also non-positive on $A$. Consequently, both integrals must be zero, forcing $|A| = 0$. Therefore, $u \geq \underline{u}$ almost everywhere in $\Omega$.
		
		\begin{lemma}\label{u_greater_u_lambda}
        For sufficiently small $\lambda > 0$, every positive solution $u$ of problem \eqref{main_singular} that satisfies $u \geq \underline{u}$ must also satisfy $u > \underline{u}$ almost everywhere in $\Omega$.
		\end{lemma}
		\begin{proof}
			Apparently, 
			$\underline{u}$ is a subsolution of \eqref{main_singular}. therefore, any solution $u>0$ of \eqref{main_singular} obeys $u\geq \underline{u}$ a.e. in $\Omega$. 
			Let 
			$$B=\{x\in\Omega:u(x)=\underline{u}(x)\}.$$ 
		$B$ being a measurable set can be approximated with closed subsets of $B$ such that for every $\delta>0,$ there exists a closed subset $F\subset B$ such that $|B\setminus F|<\delta$. We now claim that $|B|=0$. For if not, let a test function $T\in C_c^1(\Omega)$ be defined as 
			\begin{equation}T(x)=\begin{cases}
					1,& ~\text{if}~ x\in F\\
					0<T<1,&~\text{if}~x\in B\setminus F\\
					0,&~\text{if}~x\in \Omega\setminus B.
			\end{cases}\end{equation}
			Since $u$ is a weak solution of problem \eqref{main_singular}, there exist $$\theta_s, \theta_{ns}\in [\underline{g}(u),\overline{g}(u)], \\~  v_s, v_{ns}\in\partial f(x)$$
			 such that  
			\begin{align}\label{equality_breakage}
				\begin{split}
					0=&  \langle -\Delta_pu,T\rangle-\langle -\Delta_p\underline{u},T\rangle-\int_{F\cap\{u=b\}}\theta_{ns}T dS-\int_{F\cap\{u>b\}}u^{q-1}dx\\
					&+\int_{B\setminus F\cap\{u=b\}}\theta_s t dS+\int_{B\setminus F\cap\{u>b\}}u^{q-1}t dx-\int_{\partial\Omega}v_{ns}T dS+\int_{\partial \Omega}v_s T dS\\
					&+\lambda\int_{F}u^{-\gamma}dx+\lambda\int_{B\setminus F}u^{-\gamma} T dx
					=\lambda\int_{F}u^{-\gamma}dx+\lambda\int_{B\setminus F}u^{-\gamma} T dx>0,
				\end{split}
			\end{align}
			thereby leading to an absurdity. Hence $|B|=0$ and thus $u>\underline{u}$ a.e. in $\Omega$.
		\end{proof}

\begin{proof}[Proof of Theorem \ref{existence}]
The remaining part of the proof follows from an identical approach as in the proof of Theorem \ref{multiplicity}.
\end{proof}

\vskip 0.5cm

\subsection*{Acknowledgements}  Part of this research was done during the visit by the first author to the Institute of Mathematics, Physics and Mechanics in Ljubljana, supported by the Slovenian Research and Innovation Agency grant P1-0292.   The authors thank all referees for their constructive comments. The first author also thanks Dr. S.R. Pattanaik, National Institute of Technology Rourkela, for the useful discussions on nonsmooth analysis.
 
\subsection*{Research Funding}   The first author was supported by the N.B.H.M., Department of Atomic Energy (DAE) India, [02011/47/2021/NBHM(R.P.)/R \& D II/2615]. The second author was supported by the Slovenian Research and Innovation Agency program P1-0292 and grants J1-4001, J1-4031, and N1-0278.

\end{document}